\documentclass[12pt]{amsart}
\usepackage{amsfonts,amssymb,amscd,amsmath,enumerate,verbatim, graphicx, color}
\usepackage[latin1]{inputenc}
\usepackage{amscd}
\usepackage{latexsym}
\usepackage{url}
\usepackage{mathptmx}
\usepackage{multicol}
\input xy
\xyoption{all}
\def\frk{\mathfrak}               % font for "Fraktur"

\def\Phi{{\frk N}}
\def\opn#1#2{\def#1{\operatorname{#2}}} % to make operators
\opn\chara{char} 
\opn\length{\ell} 
\opn\pd{pd} 
\opn\rk{rk}
\opn\projdim{proj\,dim} 
\opn\injdim{inj\,dim} 
\opn\rank{rank}
\opn\depth{depth} 
\opn\grade{grade} 
\opn\height{height}
\opn\embdim{emb\,dim} 
\opn\codim{codim}

\opn\Tr{Tr} 
\opn\bigrank{big\,rank}
\opn\superheight{superheight}
\opn\lcm{lcm}
\opn\trdeg{tr\,deg}%\emph{
\opn\reg{reg} 
\opn\lreg{lreg} 
\opn\ini{in} 
\opn\lpd{lpd}
\opn\size{size}
\opn\mult{mult}
\opn\dist{dist}
\opn\cone{cone}
\opn\lex{lex}
\opn\rev{rev}
\opn\im{im}
\opn\m{m}
\opn\div{div} \opn\Div{Div} \opn\cl{cl} \opn\Cl{Cl}
\opn\Spec{Spec} \opn\Supp{Supp} \opn\supp{supp} \opn\Sing{Sing}
\opn\Ass{Ass} \opn\Min{Min}
\opn\Ann{Ann} \opn\Rad{Rad} \opn\Soc{Soc}
\opn\Syz{Syz} \opn\Im{Im} \opn\Ker{Ker} \opn\Coker{Coker}
\opn\Am{Am} \opn\Hom{Hom} \opn\Tor{Tor} \opn\Ext{Ext}
\opn\End{End} \opn\Aut{Aut} \opn\id{id} \opn\ini{in}

\opn\nat{nat}
\opn\pff{pf}%   \pf exists already
\opn\Pf{Pf} \opn\GL{GL} \opn\SL{SL} \opn\mod{mod} \opn\ord{ord}
\opn\Gin{Gin}
\opn\Hilb{Hilb}\opn\adeg{adeg}\opn\std{std}\opn\ip{infpt}
\opn\Pol{Pol}
\opn\sat{sat}
\opn\Var{Var}
\opn\Gen{Gen}

\opn\aff{aff} \opn\con{conv} \opn\relint{relint} \opn\st{st}
\opn\lk{lk} \opn\cn{cn} \opn\core{core} \opn\vol{vol}
\opn\link{link} \opn\star{star}
\opn\gr{gr}

\def\pot#1#2{#1[\kern-0.28ex[#2]\kern-0.28ex]}

\opn\dirlim{\underrightarrow{\lim}}
\opn\inivlim{\underleftarrow{\lim}}
\def\Implies{\ifmmode\Longrightarrow \else
        \unskip${}\Longrightarrow{}$\ignorespaces\fi}
\def\implies{\ifmmode\Rightarrow \else
        \unskip${}\Rightarrow{}$\ignorespaces\fi}
\def\iff{\ifmmode\Longleftrightarrow \else
        \unskip${}\Longleftrightarrow{}$\ignorespaces\fi}

\let\:=\colon
\newtheorem{Theorem}{Theorem}[section]
\newtheorem{Lemma}[Theorem]{Lemma}
\newtheorem{Corollary}[Theorem]{Corollary}

\newtheorem{Remark}[Theorem]{Remark}

\newtheorem{Example}[Theorem]{Example}

\newtheorem{Definition}[Theorem]{Definition}

\newtheorem{Question}[Theorem]{Question}

\newtheorem{Notation}[Theorem]{Notation}
\newtheorem{Fact}{Fact}
\newtheorem{MainTheorem}{Main Theorem}

\let\epsilon\varepsilon
\let\phi=\varphi
\let\kappa=\varkappa
\def\qed{\ifhmode\textqed\fi
      \ifmmode\ifinner\quad\qedsymbol\else\dispqed\fi\fi}
\def\textqed{\unskip\nobreak\penalty50
       \hskip2em\hbox{}\nobreak\hfil\qedsymbol
       \parfillskip=0pt \finalhyphendemerits=0}
\def\dispqed{\rlap{\qquad\qedsymbol}}

\opn\dis{dis}
\opn\height{height}
\opn\dist{dist}
\def\pnt{{\raise0.5mm\hbox{\large\bf.}}}

\opn\Lex{Lex}

\begin{document}

\title{The minimum number of vertices and edges of connected graphs with ${\rm ind}\mathchar`-{\rm match}(G) = p$, ${\rm min}\mathchar`-{\rm match}(G) = q$ and ${\rm match}(G) = r$}

%first author
\author{Kazunori Matsuda}
\address{Kazunori Matsuda,
Kitami Institute of Technology, 
Kitami, Hokkaido 090-8507, Japan}
\email{kaz-matsuda@mail.kitami-it.ac.jp}

%second author
\author{Ryosuke Sato}
\address{Ryosuke Sato,
Kitami Institute of Technology, 
Kitami, Hokkaido 090-8507, Japan}
\email{sasuke09.sa@gmail.com}

%third author
\author{Yuichi Yoshida}
\address{Yuichi Yoshida,
Kitami Institute of Technology, 
Kitami, Hokkaido 090-8507, Japan}
\email{yosuga.1214@gmail.com}

% \thanks{
% }
\begin{abstract} 
Let ${\rm ind}\mathchar`-{\rm match}(G)$, ${\rm min}\mathchar`-{\rm match}(G)$ and ${\rm match}(G)$ denote the induced matching number, minimum matching number and matching number of a graph $G$, respectively. 
It is known that ${\rm ind}\mathchar`-{\rm match}(G) \leq {\rm min}\mathchar`-{\rm match}(G) \leq {\rm match}(G) \leq 2{\rm min}\mathchar`-{\rm match}(G)$ holds. 
In the present paper, we investigate the minimum number of vertices and edges of connencted simple graphs $G$ with ${\rm ind}\mathchar`-{\rm match}(G) = p$, ${\rm min}\mathchar`-{\rm match}(G) = q$ and ${\rm match}(G) = r$ for pair of integers $p, q, r$ such that $1 \leq p \leq q \leq r \leq 2q$. 
\end{abstract}

\maketitle

\section*{Introduction}  

In this paper, we assume that all graphs are finite and simple. 
Let us recall that a graph is simple if it is undirected graph containing no loops or multiple edges. 
We denote by $|X|$ the cardinality of a finite set $X$.

Let $G$ be a graph on the vertex set $V(G)$ with the edge set $E(G)$. 
We first recall the definitions of matching number, minimum matching number and induced matching number, defined by the notions of matching, maximal matching and induced matching, respectively.  

\begin{itemize} 
	\item A subset $M \subset E(G)$ is said to be a {\em matching} of $G$ if $e \cap f = \emptyset$ for all $e, f \in M$ with $e \neq f$.  
	\item A {\em maximal matching} of $G$ is a matching $M$ of $G$ for which $M \cup \{e\}$ is not a matching of $G$ for all $e \in E(G) \setminus M$. 
	\item A matching $M$ of $G$ is said to be an {\em induced matching} if, for all $e, f \in M$ with $e \neq f$, there is no edge $g \in E(G)$ with $e \cap g \neq \emptyset$ and $f \cap g \neq \emptyset$. 
	\item The {\em matching number} match$(G)$, the {\em minimum matching number} min-match$(G)$ and the {\em induced matching number} ind-match$(G)$ of $G$ are defined as follows respectively:
	\begin{eqnarray*}
	\text{match}(G)&=&\max\{|M| : M \text{\ is\ a\ matching\ of\ } G \}; \\
	\text{min-match}(G)&=&\min\{|M| : M \text{\ is\ a\ maximal\ matching\ of\ } G \}; \\
	\text{ind-match}(G)&=&\max\{|M| : M \text{\ is\ an\ induced\ matching\ of\ } G \}.  
	\end{eqnarray*}
\end{itemize}

We remark that it is pointed out that the minimum matching number of $G$ is equal to the edge domination number of $G$ in \cite[Chapter 10]{Harary1969}, see also \cite{Chaemchan2010,YannakakisGavril1980}. 

In general, the following inequalities 

\[
1 \leq {\rm ind}\mathchar`-{\rm match}(G) \leq {\rm min}\mathchar`-{\rm match}(G) \leq {\rm match}(G) \leq 2{\rm min}\mathchar`-{\rm match}(G)
\]
and 
\[
2{\rm match}(G) \leq |V(G)|
\]
hold for all graphs $G$ with $E(G) \neq \emptyset$ (see \cite{HHKT}). 
The equality $2{\rm match}(G) = |V(G)|$ holds if and only if $G$ has a perfect matching and 
some characterizations of graphs which have a perfect matching are known (see \cite{HassaniMonfared-Mallik2016},  \cite{Tutte1947}). 
\cite[Corollary 2]{Sumner} says that $2{\rm match}(G) = |V(G)|$ holds if $G$ is a connected $K_{1, 3}$-free graph such that $|V(G)|$ is even. 
In \cite[Theorem 2.1]{AV}, a characterization of 
connected graphs $G$ with $2{\rm min}\mathchar`-{\rm match}(G) = 2{\rm match}(G) = |V(G)|$ is given. 
In \cite{DK}, graphs $G$ with ${\rm min}\mathchar`-{\rm match}(G) = {\rm match}(G)$ or ${\rm match}(G) = 2{\rm min}\mathchar`-{\rm match}(G)$ are investigated. 
A characterization of connected graphs $G$ with  ${\rm ind}\mathchar`-{\rm match}(G) = {\rm min}\mathchar`-{\rm match}(G) = {\rm match}(G)$ is given (\cite[Theorem 1]{CW}, \cite[Remark 0.1]{HHKO}). 
These graphs are also of interest in the view of combinatorial commutative algebra (cf.  \cite{BVMVT2018,FaridiMadduweHewalage,FicarraMoradi,GLW2022,HerzogHibiMoradi2022,HHKO,HKKMVT,HKMT,HKMVT,HibiMoradi2024,SF-EJC,SF-CA,SF-JPAA,SF-JA,Trung2020}). 
In \cite{HHKT}, a characterization of connected graphs $G$ with ${\rm ind}\mathchar`-{\rm match}(G) = {\rm min}\mathchar`-{\rm match}(G)$ is given. 
By definition of induced matching, one has ${\rm ind}\mathchar`-{\rm match}(G) = 1$ if and only if $G$ is $2K_{2}$-free. 
Since threshold graphs, split graphs, domishold graphs and difference graphs are $2K_{2}$-free (\cite[Theorem 1]{ChvatalHammer1977}, \cite{FoldesHammer1977},  \cite{BenzakenHammer1978}, \cite[Proposition 2.6]{HammerPeledSun1990}), the induced matching number of these graphs are one.  

Our study comes from the following two facts. 

\begin{Fact}[\cite{HHKT}, see also \cite{HM,MY}]\normalfont
For all integers $p, q$ and $r$ with $1 \leq p \leq q \leq r \leq 2q$, there exists a connected graph $G = G(p, q, r)$ such that ${\rm ind}\mathchar`-{\rm match}(G) = p$, ${\rm min}\mathchar`-{\rm match}(G) = q$ and ${\rm match}(G) = r$. 
\end{Fact}

\begin{Fact}\normalfont 
The existence of connected graphs $G$ with ${\rm ind}\mathchar`-{\rm match}(G) = p$, ${\rm min}\mathchar`-{\rm match}(G) = q$ and ${\rm match}(G) = r$ is not unique. 
In particular, the number of vertices and edges of such graphs can take various values. 
For example, let us consider the complete graph $K_{4}$ and the cycle graph $C_{5}$. 
Then we have
\[
{\rm ind}\mathchar`-{\rm match}(K_{4}) = {\rm ind}\mathchar`-{\rm match}(C_{5}) = 1, \ \ {\rm min}\mathchar`-{\rm match}(K_{4}) = {\rm min}\mathchar`-{\rm match}(C_{5}) = 2, 
\]
\[
{\rm match}(K_{4}) = {\rm match}(C_{5}) = 2 
\]
but $\left( |V(K_{4})|, |E(K_{4})| \right) = (4, 6) \ \text{and} \ \left( |V(C_{5})|, |E(C_{5})| \right) = (5, 5)$. 
\end{Fact}

Based on the above facts, we investigate $\min\left( p, q, r; |V| \right)$ and $\min\left( p, q, r; |E| \right)$, which are defined as follows: 

\begin{Definition}
For $p, q, r \in \mathbb{Z}$ with $1 \leq p \leq q \leq r \leq 2q$, we define 
\begin{eqnarray*}
\min\left( p, q, r; |V| \right) 
&=& \left\{ |V(G)| ~\left|~
\begin{array}{c}
  \mbox{$G$\rm{\ is\ a\ connected\ graph}\ \rm{with}\ $\text{\rm ind-match}(G) = p,$} \\
  \mbox{$\text{\rm min-match}(G) = q \ {\rm and} \ \text{\rm match}(G) = r$} \\ 
\end{array}
\right \}\right. ,
\end{eqnarray*}
\begin{eqnarray*}
\min\left( p, q, r; |E| \right) 
&=& \left\{ |E(G)| ~\left|~
\begin{array}{c}
  \mbox{$G$\rm{\ is\ a\ connected\ graph}\ \rm{with}\ $\text{\rm ind-match}(G) = p,$} \\
  \mbox{$\text{\rm min-match}(G) = q \ {\rm and} \ \text{\rm match}(G) = r$} \\ 
\end{array}
\right \}\right. .
\end{eqnarray*}
\end{Definition}

Now we describe our main theorems in the present paper. 

\begin{MainTheorem}[Theorem 2.1]
Let $p, q, r$ be integers with $1 \leq p \leq q \leq r \leq 2q$. 
Then
\begin{enumerate}
    \item[$(1)$] $\min\left( 1, q, r; |V| \right) = 2r$. 
    \item[$(2)$] $\min\left( p, q, r; |V| \right) = 2r$ \ if \ $2 \leq p \leq q < r \leq 2q$. 
    \item[$(3)$] $\min\left( p, r, r; |V| \right) = 2r + 1$ \ if \ $2 \leq p \leq q = r$. 
\end{enumerate}
\end{MainTheorem}

\begin{MainTheorem}[Theorem 3.1]
Let $p, q, r$ be integers with $1 \leq p \leq q \leq r \leq 2q$. Then
\begin{enumerate}
    \item[$(1)$] $\min\left( 1, q, 2q; |E| \right) = \binom{2q + 1}{2}$ and $\min\left( 1, q, 2q - 1; |E| \right) = \binom{2q}{2}$ for all $q \geq 1$. 
    \item[$(2)$] $\min\left( q, q, r; |E| \right) = 2r - 1$ \ if \ $2 \leq p = q < r \leq 2q$. 
    \item[$(3)$] $\min\left( r, r, r; |E| \right) = 2r$ \ if \ $2 \leq p = q = r$. 
\end{enumerate}
\end{MainTheorem}

\begin{MainTheorem}[Theorem 3.2]
Let $q, r$ be integers with $2 \leq q \leq r \leq 2q - 2$. Then
\begin{enumerate}
    \item[$(1)$] $\min\left( 1, q, q; |E| \right) \leq q^2$. 
    \item[$(2)$] $\min\left( 1, q, q + 1; |E| \right) \leq q^2 + 2$. 
    \item[$(3)$] $\min\left( 1, q, r; |E| \right) \leq \min\{ f_{1}(q, r), f_{2}(q, r) \}$ \ if \ $q + 2 \leq r \leq 2q - 2$, where 
    \[
    f_{1}(q, r) = r(q - 1) + \binom{r - q + 2}{2}, \ \ f_{2}(q, r) = 2(r - q) + \binom{2q}{2}. 
    \]
\end{enumerate}    
\end{MainTheorem}

\begin{MainTheorem}[Theorem 3.4]
Let $p, q, r$ be integers with $2 \leq p < q \leq r \leq 2q$. Then
\begin{enumerate}
    \item[$(1)$] $\min\left( p, q, q; |E| \right) \leq (a_{1}^{2} + 1)p + (2a_{1} + 1)b_{1}$, where $a_{1}$ and $b_{1}$ are non-negative integers such that $q = a_{1}p + b_{1}$ and $0 \leq b_{1} \leq p-1$.  
    \item[$(2)$] If $q < r \leq 2q - p + 1$, we have 
    \[
    \min\left( p, q, r; |E| \right) \leq a_{2}^{2}(p - 1) + (2a_{2} + 1)b_{2} + p + \binom{2(r - q) + 1}{2}, 
    \]
    where $a_{2}$ and $b_{2}$ are non-negative integers such that $2q - r = a_{2}(p - 1) + b_{2}$ and $0 \leq b_{2} \leq p-2$.
    \item[$(3)$] If $2q - p + 1 < r \leq 2q$, we have
    \[
    \min\left( p, q, r; |E| \right) \leq p + 2q - r + (p - 2q + r)\binom{2a_{3}+1}{2} + b_{3}(4a_{3} + 3) , 
    \]
    where $a_{3}$ and $b_{3}$ are non-negative integers such that $r - q = a_{3}(p - 2q + r) + b_{3}$ and $0 \leq b_{3} \leq p-2q + r - 1$.
\end{enumerate}    
\end{MainTheorem}

\begin{Notation}\normalfont 
We summarize our notations here.  
\begin{itemize}
%    \item We denote an edge connecting vertices $u$ and $v$ by $\{u, v\}$. 
%    \item We denote by $|X|$ the cardinality of a finite set $X$. 
    \item For a non-negative integer $m$, we define $X_{m} = \{ x_1, x_2, \ldots, x_{m} \}$. 
    Note that $X_{0} = \emptyset$. 
    In the same way, we also define $Y_{m}$, $Z_{m}$, $U_{m}$, $V_{m}$ and $W_{m}$. 
    \item For a matching $M \subset E(G)$ , we write $V(M) = \{v \in V(G) \mid v \in e \ \text{for some} \ e \in M\}$. 
    $M$ is said to be a {\em perfect matching} of $G$ if $V(M) = V(G)$. 
    \item For $v \in V(G)$, we write $N_{G}(v) = \{ w \in V(G) \mid \{v, w\} \in E(G) \}$, $N_{G}[v] = N_{G}(v) \cup \{v\}$ and $\deg(v) = |N_{G}(v)|$.  
%    We remark that $G$ has a perfect matching if and only if $|V(G)| = 2{\rm match}(G)$. 
\end{itemize}
\end{Notation}

%%%%%%%%%%%%%%%%%%%%%%%%%%%%%%%%%%%%%%%%
% Section 1
%%%%%%%%%%%%%%%%%%%%%%%%%%%%%%%%%%%%%%%%

\section{Preparation}

In this section, we prepare for proving our main theorems. 

\subsection{Induced subgraphs and disconnected graphs}
Let $G$ be a graph and let $W$ be a subset of $V(G)$. 
The {\em induced subgraph} of $G$ on $W$, denoted by $G[W]$, is defined by:  
\begin{itemize}
	\item $V(G[W]) = W$. 
	\item $E(G[W]) = \{ \{u, v\} \in E(G) \mid u, v \in W \}$. 
\end{itemize}

\begin{Lemma}\label{induced-subgraph}
Let $G$ be a graph. 
For $W \subset V(G)$, one has  
\begin{enumerate}
	\item[$(1)$] ${\rm ind}$-${\rm match}(G[W]) \leq {\rm ind}$-${\rm match}(G)$.
	\item[$(2)$] ${\rm min}$-${\rm match}(G[W]) \leq {\rm min}$-${\rm match}(G)$.  
	\item[$(3)$] ${\rm match}(G[W]) \leq {\rm match}(G)$.  
\end{enumerate}
\end{Lemma}

\begin{Lemma}[cf. {\cite[Lemma 1.6]{MY}}]\label{disconnected}
Let $G$ be a disconnected graph and $H_{1}, \ldots, H_{s}$ $(s \geq 2)$ the connected components of $G$. 
Then  
\begin{enumerate}
	\item[$(1)$] $\displaystyle {\rm{ind}}\mathchar`-{\rm{match}}(G) = \sum_{i = 1}^{s} {\rm{ind}}\mathchar`-{\rm{match}}(H_{i})$. 
	\item[$(2)$] $\displaystyle {\rm{min}}\mathchar`-{\rm{match}}(G) = \sum_{i = 1}^{s} {\rm{min}}\mathchar`-{\rm{match}}(H_{i})$.
	\item[$(3)$] $\displaystyle {\rm{match}}(G) = \sum_{i = 1}^{s} {\rm{match}}(H_{i})$.
\end{enumerate}
\end{Lemma}

\subsection{Independent set and independence number}

Let $G$ be a graph. 
A subset $S \subset V(G)$ is said to be an {\em independent set} of $G$ if $\{v, w\} \not\in E(G)$ for all $v, w \in S$ with $v \neq w$. 
Note that the empty set $\emptyset$ and a singleton $\{v\} \subset V(G)$ are independent sets. 
The {\em independence number} of $G$, denoted by $\alpha(G)$, is defined by 
\[
\alpha(G) = \max\{ |S| : S \ \text{is\ an\ independent\ set\ of}\ G \}.
\]

\begin{Lemma}\label{max-matching}
Let $M \subset E(G)$ be a matching of a graph $G$. 
Then $M$ is a maximal matching of $G$ if and only if \ $V(G) \setminus V(M)$ is an independent set of $G$. 
\end{Lemma}
\begin{proof}
First, we assume that $M$ is a maximal matching. 
If $V(G) \setminus V(M)$ is not an independent set, there exist $v, w \not\in V(M)$ such that $\{v, w\} \in E(G)$. 
Then $M \cup \{v, w\}$ is a matching of $G$, but this is a contradiction. 
Hence $V(G) \setminus V(M)$ is an independent set. 

Next, assume that $M$ is not a maximal matching. 
Then there exists $\{v, w\} \in E(G) \setminus M$ such that $M \cup \{v, w\}$ is a matching of $G$. 
Since $v, w \in V(G) \setminus V(M)$ and $\{v, w\} \in E(G)$, $V(G) \setminus V(M)$ is not an independent set.  
\end{proof}

\begin{Lemma}\label{ind1-3}
Let $G$ be a connected graph with ${\rm ind}$-${\rm match}(G) \geq 2$. 
Let $I \subset E(G)$ be an induced matching of $G$ with $|I| \geq 2$. 
If $v \in V(I)$, then $V(G) \setminus N_{G}[v]$ is not an independent set.
\end{Lemma}
\begin{proof}
Let $\{v, w\}, \{x, y\} \in I$. 
Then $\{v, x\}, \{v, y\} \not\in E(G)$ since $I$ is an induced matching of $G$.  
Hence $x, y \in V(G) \setminus N_{G}[v]$. 
Thus we have the desired conclusion. 
\end{proof}

\begin{Lemma}\label{bounds-for-indep}
Let $M$ be a maximal matching of a connected graph $G$. 
Let $\alpha(G)$ be the independence number of $G$. 
Then 
\begin{enumerate}
	\item[$(1)$] $\alpha(G) \geq |V(G)| - 2|M|$. 
%	\item[$(2)$] $\alpha(G) \leq |V(G)| - |M|$. 
    \item[$(2)$] $|V(G)| - \alpha(G) \leq 2{\rm min}$-${\rm match}(G)$. 
\end{enumerate}
\end{Lemma}
\begin{proof}
%Let $M = \bigcup_{i = 1}^{r} \{ x_{i}, y_{i} \}$ be a maximal matching of $G$ with $|M| = r$. 
(1) Since $V(G) \setminus V(M)$ is an independent set of $G$ by Lemma \ref{max-matching}, we have 
\[
\alpha(G) \geq |V(G) \setminus V(M)| = |V(G)| - |V(M)| = |V(G)| - 2|M|.
\]

%(2) Put $S = \left( V(G) \setminus V(M) \right) \cup \{x_{1}, \ldots, x_{r}\}$. 
%Then $S$ is a maximal subset of $V(G)$ which is possible to be an independent set. 
%Hence one has 
%\[
%\alpha(G) \leq |S| = |\left( V(G) \setminus V(M) \right) \cup \{x_{1}, \ldots, x_{r}\}| = |V(G)| - |M|. 
%\]

(2) Let $M_{1}$ be a maximal matching of $G$ with $|M_{1}| = {\rm min}$-${\rm match}(G)$. 
Applying (1) for $M_{1}$, we have $|V(G)| - \alpha(G) \leq 2{\rm min}$-${\rm match}(G)$.   
%Next, let $M_{2}$ be a matching of $G$ with $|M_{2}| = \text{match}(G)$. 
%Note that the matching $M_{2}$ is maximal. 
%Applying (2) for $M_{2}$, one has ${\rm match}(G) \leq |V(G)| - \alpha(G)$. 
%Therefore we have the desired conclusion.  
\end{proof}

\subsection{The conditions $(*_1)$ and $(*_2)$}

In this subsection, we introduce two conditions $(*_1)$ and $(*_2)$ for vertices. 

\begin{Definition}\label{condition*}
Let $G$ be a connected graph. 
\begin{enumerate}
    \item[$(1)$] We say that $v \in V(G)$ satisfies the condition $(*_1)$ if there exists $w \in V(G)$ such that $\{v, w\} \in E(G)$ and $\deg(w) = 1$. 
    \item[$(2)$] We say that $v \in V(G)$ satisfies the condition $(*_2)$ if $v \in V(M)$ for all maximal matching $M$ with $|M| = {\rm min}$-${\rm match}(G)$. 
\end{enumerate}
\end{Definition}
Note that the condition $(*_1)$ means ``there exists $v \in V(G)$ such that $v$ is an incident to some leaf edge", and $(*_2)$ means ``there exists $v \in V(G)$ such that $v$ is contained in all minimum maximal matchings of $G$". 

\begin{Remark}\label{condition*_remark}
\begin{enumerate}
    \item[$(1)$] If $v \in V(G)$ satisfies $(*_1)$, then $v$ is contained in all maximal matching of $G$. 
In particular, $v$ satisfies $(*_2)$. 
    \item[$(2)$] Every vertex of the complete graph $K_{2n}$ satisfies $(*_2)$. Moreover, every vertex of the complete bipartite graph $K_{n, n}$ satisfies $(*_2)$. 
%    \item[$(3)$] Let $C_{n}$ be the cycle graph $(n \geq 3)$. Then there exists $v \in V(C_n)$ which satisfies $(*_2)$ if and only if $n = 4$. 
\end{enumerate}    
\end{Remark}

\begin{Lemma}\label{ind1-2}
Let $G$ be a connected graph with $\rm{ind}$-$\rm{match}(G) = 1$. 
Assume that there exists $v\in V(G)$ which satisfies the condition $(*_1)$. 
Then $V(G) \setminus N_{G}[v]$ is an independent set.
\end{Lemma}
\begin{proof}
By assumption, there exists $w \in V(G)$ such that $\{v, w\} \in E(G)$ and $\deg(w) = 1$. 
Suppose that $V(G) \setminus N_{G}[v]$ is not an independent set. 
Then there exists $\{x, y\} \in E(G)$ with $x, y \not\in N_{G}[v]$. 
Since $\{x, v\}, \{y, v\}, \{x, w\}, \{y, w\} \not\in E(G)$, $\left\{ \{x, y\}, \{v, w\} \right\}$ is an induced matching of $G$, but this contradicts for ind-match$(G) = 1$.  
\end{proof}

Theorems \ref{Thm-ind} and \ref{Thm-min} are used for proving Theorem \ref{4thMainThm} in Section 3. 

\begin{Theorem}\label{Thm-ind}
Let $H_{1}, \ldots, H_{s}$ ($s \geq 2$) be connected graphs. 
Assume that for each $1 \leq i \leq s$, {\rm ind}-{\rm match}$(H_{i}) = 1$ and 
$H_{i}$ satisfies one of the following conditions:
\begin{enumerate}
    \item[$(a)$] There exists $v_{i} \in V(H_{i})$ which satisfies the condition $(*_{1})$. 
    \item[$(b)$] $H_{i}$ is a complete bipartite graph. 
\end{enumerate}
In the case that $H_{i}$ satisfies $(b)$, we take $v_{i} \in V(H_{i})$ arbitrary. 
Let $G$ be the graph with 
\[
V(G) = \left\{ \bigcup_{i = 1}^{s} V(H_{i}) \right\} \ \cup \ \{v\} \ \ \text{and} \ \ E(G) = \left\{ \bigcup_{i = 1}^{s} E(H_{i}) \right\} \ \cup \ \left\{ \{v_{i}, v\} \mid 1 \leq i \leq s \right\}, 
\]
where $v$ is a new vertex. 
Then {\rm ind}-{\rm match}$(G) = s$. 
\end{Theorem}
\begin{proof}
First, we can see that ind-match$(G) \geq s$ by Lemma \ref{induced-subgraph} and \ref{disconnected} because $\bigcup_{i = 1}^{s} H_{i}$ is an induced subgraph of $G$. 
Hence it is enough to show ind-match$(G) \leq s$. 
Suppose that ind-match$(G) > s$. 
Then we can take an induced matching $I$ of $G$ with $|I| > s$. 
If $v \not\in V(I)$, then $I$ is also an induced matching of $\bigcup_{i = 1}^{s} H_{i}$, but this contradicts ind-match$\left( \bigcup_{i = 1}^{s} H_{i} \right) = s$. 
Thus we have $v \in V(I)$. 
We may assume that $\{v_{1}, v\} \in I$. 
Note that $v_{1} \in V(H_{1})$ and $H_{1}$ satisfies either $(a)$ or $(b)$. 
\begin{itemize}
    \item We consider the case that $H_{1}$ satisfies $(a)$. 
    Then $V(H_{1}) \setminus N_{H_{1}}[v_{1}]$ is an independent set from Lemma \ref{ind1-2}. This fact means $e \not\in I$ for all $e \in E(H_{1})$. 
    \item We consider the case that $H_{1}$ satisfies $(b)$. 
    Then $H_{1}$ is a complete bipartite graph with bipartition $V(H_{1}) = X \cup Y$. 
    We may assume $v_{1} \in X$. 
    Since $V(H_{1}) \setminus N_{H_{1}}[v_{1}] = X \setminus \{v_{1}\}$ is an independent set, it follows that $e \not\in I$ for all $e \in E(H_{1})$. 
\end{itemize}
In either case, $|I \cap E(H_{1})| = 0$ holds. 
Since $\{v_{1}, v\} \in I$, one has $\{v_{i}, v\} \not\in I$ for all $2 \leq i \leq s$. 
Hence $I = \{v_{1},v \} \ \cup \ \left[ \bigcup_{i = 1}^{s} \left\{ I \cap E(H_{i}) \right\} \right]$. 
Thus we have 
\[
s + 1 \leq |I| = 1 + \sum_{i = 1}^{s} |I \cap E(H_{i})| = 1 + \sum_{i = 2}^{s} |I \cap E(H_{i})|. 
\]
By pigeonhole principle, there exists $2 \leq j \leq s$ with $|I \cap E(H_{j})| \geq 2$. 
However, this is a contradiction because $I \cap E(H_{j})$ is an induced matching of $H_{j}$ but ind-match$(H_{j}) = 1$. 
Therefore we have ind-match$(G) = s$. 
\end{proof}

\begin{Theorem}\label{Thm-min}
Let $H_{1}, \ldots, H_{s}$ ($s \geq 2$) be connected graphs. 
We assume that for each $1 \leq i \leq s$, there exists $v_{i} \in V(H_{i})$ which satisfies the condition $(*_{2})$. 
%Let $G$ be the graph on the vertex set $V(G) = \left\{ \bigcup_{i = 1}^{s} V(H_{i}) \right\} \ \cup \ \{v\}$ with $E(G) = \left\{ \bigcup_{i = 1}^{s} E(H_{i}) \right\} \ \cup \ \left\{ \{v_{i}, v\} \mid 1 \leq i \leq s \right\}$, where $v$ is a new vertex. 
Let $G$ be the graph which appears in Theorem \ref{Thm-ind}. 
Then ${\rm min}$-${\rm match}(G) = \sum_{i = 1}^{s} {\rm min}$-${\rm match}(H_{i})$. 
\end{Theorem}
\begin{proof}
For each $1 \leq i \leq s$, let $M_{i}$ be a maximal matching of $H_{i}$ with $|M_{i}| = {\rm min}$-${\rm match}(H_{i})$. 
Then $\bigcup_{i = 1}^{s} M_{i}$ is a maximal matching of the disconnected graph $\bigcup_{i = 1}^{s} H_{i}$. 
Since $\bigcup_{i = 1}^{s} H_{i}$ is an induced subgraph of $G$ on $\bigcup_{i = 1}^{s} V(H_{i})$, it follows that 
\[
{\rm min}\mathchar`-{\rm match}(G) \geq {\rm min}\mathchar`-{\rm match}\left( \bigcup_{i = 1}^{s} H_{i} \right) = \sum_{i = 1}^{s} {\rm min}\mathchar`-{\rm match}(H_{i})
\]
from Lemma \ref{induced-subgraph} and \ref{disconnected}. 
Moreover, one has $v_{i} \in V(M_{i})$ for all $1 \leq i \leq s$ by assumption.  
Hence 
\[
V(G) \setminus V\left( \bigcup_{i = 1}^{s} M_{i} \right) \ = \ V(G) \setminus \left\{ \bigcup_{i = 1}^{s} V(M_{i}) \right\} \ = \ \left[ \bigcup_{i = 1}^{s} \left\{ V(G_{i}) \setminus V(M_{i}) \right\} \right] \ \cup \ \{v\} 
\]
is an independent set of $G$. 
Thus $\bigcup_{i = 1}^{s} M_{i}$ is a maximal matching of $G$ by Lemma \ref{max-matching}. 
Therefore we have 
\[{\rm min}\mathchar`-{\rm match}(G) \leq \left| \bigcup_{i = 1}^{s} M_{i} \right| = \sum_{i = 1}^{s} {\rm min}\mathchar`-{\rm match}(H_{i}).
\]
Hence it follows that min-match$(G) = \sum_{i = 1}^{s}$ min-match$(H_{i})$. 
\end{proof}

\subsection{Lower bounds for $\min\left( p, q, r; |E| \right)$} 

In this subsection, we give some lower bounds for $\min\left( p, q, r; |E| \right)$. 

\begin{Lemma}\label{ind1-1}
Let $G$ be a connected graph with $\rm{ind}$-${\rm match}(G) = 1$. 
Then one has $|E(G)| \geq \binom{{\rm match}(G) + 1}{2}$. 
In particular, $\min\left( 1, q, r; |E| \right) \geq \binom{r + 1}{2}$. 
\end{Lemma}
\begin{proof}
Let match$(G) = r$ and let $\{e_{1}, \ldots, e_{r}\}$ be a matching of $G$. 
Since ind-match$(G) = 1$, for all $1 \leq i < j \leq r$, there exists $f_{ij} \in E(G)$ such that $e_{i} \cap f_{ij} \neq \emptyset$ and $e_{j} \cap f_{ij} \neq \emptyset$. 
Note that $f_{ij} \neq f_{k\ell}$ if $(i, j) \neq (k, \ell)$. 
Hence one has $|E(G)| \geq r + \binom{r}{2} = \binom{r + 1}{2}$. 
\end{proof}

\begin{Lemma}\label{minE-minV}
Let $p, q, r$ be integers with $1 \leq p \leq q \leq r \leq 2q$. 
Then $\min\left( p, q, r; |E| \right) \geq \min\left( p, q, r; |V| \right) - 1$. 
\end{Lemma}
\begin{proof}
Let $\min\left( p, q, r; |V| \right) = s$. 
Then $|V(G)| \geq s$ holds for all connected graph $G$ with ind-match$(G) = p$, min-match$(G) = q$ and match$(G) = r$. 
Since $|E(G)| \geq |V(G)| - 1$, one has 
$\min\left( p, q, r; |E| \right) \geq s - 1 = \min\left( p, q, r; |V| \right) - 1.$
\end{proof}

\subsection{The graph $G_{r}$}\label{1.5} 
Let $r \geq 2$ be an integer. 
We define the graph $G_{r}$ as follows: 
\begin{itemize}
    \item $V(G_{r}) = X_{r} \cup Y_{r}$. 
    \item $E(G_{r}) = \left\{ \{x_{i}, x_{j}\} \mid 1 \leq i < j \leq r \right\} \cup \left\{ \{x_{k}, y_{k}\} \mid 1 \leq k \leq r \right\}$. 
\end{itemize}
Note that $|V(G_{r})| = 2r$ and $|E(G_{r})| = \binom{r}{2} + r = \binom{r + 1}{2}$.  

\begin{Lemma}\label{G_r}
Let $G_{r}$ be the graph as above. 
Then ${\rm ind}$-${\rm match}(G_{r}) = 1$, ${\rm min}$-${\rm match}(G_{r}) = \lceil r/2 \rceil$ and ${\rm match}(G_{r}) = r$. 
\end{Lemma}
\begin{proof} 
Since $G_{r}$ is a split graph, 
$G_{r}$ is $2K_{2}$-free by \cite{FoldesHammer1977}.   
Hence one has ${\rm ind}$-${\rm match}(G_{r}) = 1$. 
Moreover, we can see that ${\rm match}(G_{r}) = r$ since $\left\{ \{x_{k}, y_{k}\} \mid 1 \leq k \leq r \right\}$ is a perfect matching of $G_{r}$. 

Now we prove ${\rm min}$-${\rm match}(G_{r}) = \lceil r/2 \rceil$. 
Since $|V(G_{r})| = 2r$ and $\alpha(G_{r}) = r$, one has $r \leq 2{\rm min}$-${\rm match}(G_{r})$ by Lemma \ref{bounds-for-indep}(2). 
Hence ${\rm min}$-${\rm match}(G_{r}) \geq \lceil r/2 \rceil$ holds. 
If $r$ is even, then $\left\{ \{x_{i}, x_{r/2 + i} \} \mid 1 \leq i \leq r/2 \right\}$ is a maximal matching of $G_{r}$. 
If $r$ is odd, then $\{x_{1}, y_{1}\} \cup \left\{ \{x_{i}, x_{\frac{r - 1}{2} + i} \} \mid 2 \leq i \leq \frac{r + 1}{2} \right\}$ is a maximal matching of $G_{r}$. 
In any case, it follows that ${\rm min}$-${\rm match}(G_{r}) \leq \lceil r/2 \rceil$. 
Thus we have ${\rm min}$-${\rm match}(G_{r}) = \lceil r/2 \rceil$. 
\end{proof}
   
\begin{Remark}
For a graph $G$ with vertex set $V(G) = \{v_{1}, \ldots, v_{m}\}$, the whiskered graph $W(G)$ of $G$ is the graph on the vertex set $V(W(G)) = V(G) \cup \{w_{1}, \ldots w_{m}\}$ and edge set $E(W(G)) = E(G) \cup \left\{ \{v_{i}, w_{i} \} \mid 1 \leq i \leq m \right\}$. 
Note that the graph $G_{r}$ as above is the whiskered graph of $K_{r}$. 
It follows that $\alpha(W(G)) = {\rm match}(W(G)) = |V(G)|$ for all graph $G$. 
Moreover, ${\rm match}(W(G)) = 2{\rm min}\mathchar`-{\rm match}(W(G))$ holds if $G$ has a perfect matching. 
It is known that the whiskered graph also has some good properties from the perspective of combinatorial commutative algebra, see \cite{CFHNVT2024,DochtermannEngstrom2009,FranciscoHa2008,HollebenNicklasson2025,Villarreal1990,Woodroofe2009}. 

\end{Remark}

\section{$\min\left( p, q, r; |V| \right)$}

In this section, we determine $\min\left( p, q, r; |V| \right)$ for all integers $p, q, r$ with $1 \leq p \leq q \leq r \leq 2q$. 

\begin{Theorem}\label{1stMainThm}
Let $p, q, r$ be integers with $1 \leq p \leq q \leq r \leq 2q$. 
Then 
\begin{enumerate}
    \item[$(1)$] $\min\left( 1, q, r; |V| \right) = 2r$. 
    \item[$(2)$] $\min\left( p, q, r; |V| \right) = 2r$ \ if \ $2 \leq p \leq q < r \leq 2q$. 
    \item[$(3)$] $\min\left( p, r, r; |V| \right) = 2r + 1$ \ if \ $2 \leq p \leq q = r$. 
\end{enumerate}
\end{Theorem}
\begin{proof}
Let
\begin{eqnarray*}
& & {\rm{\bf Graph}_{ind\text{-}match, min\text{-}match, match}}(n) \\
&=& \left\{(p, q, r) \in \mathbb{N}^{3} ~\left|~
\begin{array}{c}
  \mbox{\rm{There\ exists\ a\ connected\ simple\ graph}\ $G$ \ \rm{with}\ $|V(G)| = n$} \\
  \mbox{{\rm{and}} \ $\text{\rm ind-match}(G) = p, \ \text{\rm min-match}(G) = q, \ \text{\rm match}(G) = r$} \\ 
\end{array}
\right \}\right. .  
\end{eqnarray*}
Note that $(p, q, r) \not\in {\rm{\bf Graph}_{ind\text{-}match, min\text{-}match, match}}(n)$ if $n < 2r$ because $|V(G)| \geq 2\rm{match}(G)$. 
By virtue of \cite[Theorem 2.1]{MY}, we have
\begin{eqnarray*}
& & {\rm{\bf Graph}_{ind\text{-}match, min\text{-}match, match}}(2r) \\ 
&=& \left\{ (1, q, r) \in \mathbb{N}^{3} \ \middle| \ 1 \leq q \leq r \leq 2q  \right\} 
\ \cup \ \left\{ (p, q, r) \in \mathbb{N}^{3} \ \middle| \ 2 \leq p \leq q < r \leq 2q \right\} 
\end{eqnarray*}
and  
\begin{eqnarray*}
& & {\rm{\bf Graph}_{ind\text{-}match, min\text{-}match, match}}(2r + 1) 
\ = \ \left\{ (p, q, r) \in \mathbb{N}^{3} \ \middle| \ 1 \leq p \leq q \leq r \leq 2q  \right\} . 
\end{eqnarray*}
Hence we have the desired conclusion. 
\end{proof}

As a corollary, we have 

\begin{Corollary}\label{notPM}
Let $G$ be a connected simple graph. 
Assume that $2 \leq \rm{ind}$-$\rm{match}(G)$ and $\rm{min}$-$\rm{match}(G) = \rm{match}(G)$. 
Then $G$ does not have any perfect matching. 
\end{Corollary}

By virtue of Theorem \ref{1stMainThm} together with Lemma \ref{minE-minV}, we also have 

\begin{Corollary}\label{LowerBound}
Let $p, q, r$ be integers with $2 \leq p \leq q \leq r \leq 2q$. 
Then 
\begin{enumerate}
%    \item[$(1)$] $\min\left( 1, q, r; |E| \right) \geq 2r - 1$. 
    \item[$(1)$] $\min\left( p, q, r; |E| \right) \geq 2r - 1$ \ if \ $2 \leq p \leq q < r \leq 2q$. 
    \item[$(2)$] $\min\left( p, r, r; |E| \right) \geq 2r$ \ if \ $2 \leq p \leq q = r$. 
\end{enumerate}    
\end{Corollary}

%Section 3
\section{$\min\left( p, q, r; |E| \right)$}

In this section, we investigate $\min\left( p, q, r; |E| \right)$. 
First, for some pairs of integers $p, q$ and $r$ with $1 \leq p \leq q \leq r \leq 2q$, we determine the value of $\min\left( p, q, r; |E| \right)$. 

\begin{Theorem}\label{2ndMainThm}
Let $p, q, r$ be integers with $1 \leq p \leq q \leq r \leq 2q$. Then
\begin{enumerate}
    \item[$(1)$] $\min\left( 1, q, 2q; |E| \right) = \binom{2q + 1}{2}$ and $\min\left( 1, q, 2q - 1; |E| \right) = \binom{2q}{2}$ for all $q \geq 1$. 
    \item[$(2)$] $\min\left( q, q, r; |E| \right) = 2r - 1$ \ if \ $2 \leq p = q < r \leq 2q$. 
    \item[$(3)$] $\min\left( r, r, r; |E| \right) = 2r$ \ if \ $2 \leq p = q = r$. 
\end{enumerate}
\end{Theorem}
\begin{proof}
(1) Let $G_{2q}$ be the graph which appears in  \ref{1.5}. 
By virtue of Lemma \ref{G_r}, it follows that ${\rm ind}$-${\rm match}(G_{2q}) = 1$, ${\rm min}$-${\rm match}(G_{2q}) = q$ and ${\rm match}(G_{2q}) = 2q$. 
Since $|E(G_{2q})| = \binom{2q + 1}{2}$, we have $\min\left( 1, q, 2q; |E| \right) \leq \binom{2q + 1}{2}$. 
Moreover, Lemma \ref{ind1-1} says that $\min\left( 1, q, 2q; |E| \right) \geq \binom{2q + 1}{2}$ holds. 
Thus it follows that $\min\left( 1, q, 2q; |E| \right) = \binom{2q + 1}{2}$. 
As the same argument works for $G_{2q - 1}$, we also have $\min\left( 1, q, 2q - 1; |E| \right) = \binom{2q}{2}$ for all $q \geq 2$. 
Finally, it is easy to see that $\min\left( 1, 1, 1; |E| \right) = 1$. 

(2) First, we show $\min\left( q, q, q + 1; |E| \right) = 2q + 1$. 
Let $G_{q}^{(1)}$ be the graph as follows; see Figure \ref{fig:G_{q}^{(1)}}: 
\begin{itemize}
    \item $\displaystyle \frac{}{} V\left( G_{q}^{(1)} \right) = X_{q - 1} \cup Y_{q - 1} \cup Z_{4}$. 
    \item $\displaystyle E\left( G_{q}^{(1)} \right) = \left\{ \bigcup_{i = 1}^{q - 1} \{x_i, y_i\} \right\} \ \cup \ \left\{ \bigcup_{i = 1}^{q - 1} \{y_i, z_2\} \right\} \ \cup \ \left\{ \bigcup_{i = 1}^{3} \{z_i, z_{i + 1}\} \right\}$. 
\end{itemize}

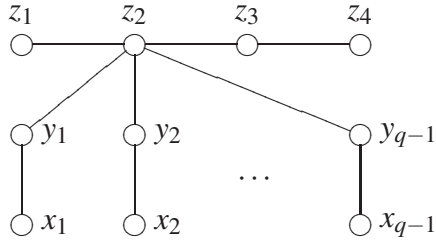
\begin{figure}[htbp]
\centering
%\bigskip

\begin{xy}
	\ar@{} (0,0);(50, -16)  *++! L{x_{1}} *\cir<4pt>{} = "X1"
	\ar@{} (0,0);(65, -16)  *++! L{x_{2}} *\cir<4pt>{} = "X2"
	\ar@{} (0,0);(95, -16)  *++! L{x_{q-1}} *\cir<4pt>{} = "Xq-1"
	\ar@{-} "X1";(50, -4) *++! L{y_{1}} *\cir<4pt>{} = "Y1";
	\ar@{-} "X2";(65, -4) *++!L{y_{2}} *\cir<4pt>{} = "Y2";
	\ar@{} (0, 0); (76, -10) *++!L{\cdots}
	\ar@{-} "Xq-1";(95, -4) *++!L{y_{q-1}} *\cir<4pt>{} = "Yq-1";
    \ar@{-} "Y1"; (65, 8) *++!D{z_{2}} *\cir<4pt>{} = "Z2";
    \ar@{-} "Z2"; (50, 8) *++!D{z_{1}} *\cir<4pt>{} = "Z1";
    \ar@{-} "Z2"; (80, 8) *++!D{z_{3}} *\cir<4pt>{} = "Z3";
    \ar@{-} "Z3"; (95, 8) *++!D{z_{4}} *\cir<4pt>{} = "Z4";
    \ar@{-} "Y2";"Z2";
    \ar@{-} "Yq-1";"Z2";
\end{xy}

  \caption{The graph $G_{q}^{(1)}$}
  \label{fig:G_{q}^{(1)}}
\end{figure}

Now we prove that ${\rm ind}\mathchar`-{\rm match}\left( G_{q}^{(1)} \right) = {\rm min}\mathchar`-{\rm match}\left( G_{q}^{(1)} \right) = q$ and ${\rm match}\left( G_{q}^{(1)} \right) = q + 1$. 
\begin{itemize}
    \item Since $\left\{ \bigcup_{i = 1}^{q - 1} \{x_i, y_i\} \right\} \ \cup \ \{z_3, z_4\}$ is an induced matching and $\left\{ \bigcup_{i = 1}^{q - 1} \{x_i, y_i\} \right\} \  \cup \ \{z_2, z_3\}$ is a maximal matching of $G_{q}^{(1)}$, we have
\[
q \leq {\rm ind}\mathchar`-{\rm match}\left( G_{q}^{(1)} \right) \leq {\rm min}\mathchar`-{\rm match}\left( G_{q}^{(1)} \right) \leq q.  
\]
    Hence ${\rm ind}\mathchar`-{\rm match}\left( G_{q}^{(1)} \right) = {\rm min}\mathchar`-{\rm match}\left( G_{q}^{(1)} \right) = q$ holds.
    \item Since $\left\{ \bigcup_{i = 1}^{q - 1} \{x_i, y_i\} \right\} \ \cup \ \left\{ \bigcup_{i = 1}^{2 \frac{}{}} \{z_{2i - 1}, z_{2i}\} \right\}$ is a perfect matching of $G_{q}^{(1)}$, one has ${\rm match}\left( G_{q}^{(1)} \right) = q + 1$.
\end{itemize} 
Thus it follows that  
\[
\min\left( q, q, q + 1; |E| \right) \leq 
\left| E\left( G_{q}^{(1)} \right) \right| = 2q + 1.
\]
By virtue of this fact together with Corollary \ref{LowerBound}(1), $\min\left( q, q, q + 1; |E| \right) = 2q + 1$ holds. 

Next, we show $\min\left( q, q, r; |E| \right) = 2r - 1$ for all $2 \leq q$ and $q + 2 \leq r \leq 2q$. 
We define the graph $G_{q, r}^{(2)}$ as follows; see Figure \ref{fig:G_{q, r}^{(2)}}:

\begin{itemize}
    \item $\displaystyle \frac{}{} V\left( G_{q, r}^{(2)} \right) = X_{2q - r + 1} \cup Y_{2q - r + 1} \cup Z_{2r - 2q + 2} \cup U_{r - q - 2} \cup V_{r - q - 2}$. 
    \item $\displaystyle E\left( G_{q, r}^{(2)} \right) = \left\{ \bigcup_{i = 1}^{2q - r + 1} \{x_i, y_i\} \right\} \ \cup \ \left\{ \bigcup_{i = 1}^{2q - r + 1} \{y_i, z_2\} \right\} \ \cup \ \left\{ \bigcup_{i = 1}^{2r - 2q + 1} \{z_i, z_{i + 1}\}  \right\}$ \\ $\displaystyle   \cup \ \left\{ \bigcup_{i = 1}^{r - q - 2} \{u_i, v_i\} \right\} \ \cup \ \left\{ \bigcup_{i = 1}^{r - q - 2} \{v_i, z_{2i + 5}\} \right\}$. 
\end{itemize}
Note that $G_{q, r}^{(2)}$ is a tree with $\left| V\left( G_{q, r}^{(2)} \right) \right| = 2r$ and $\left| E\left( G_{q, r}^{(2)} \right) \right| = 2r - 1$. 

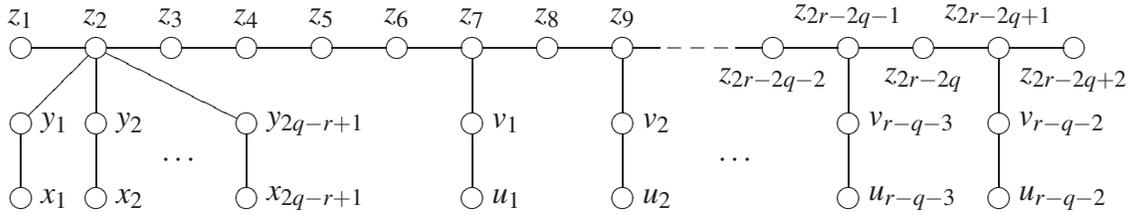
\begin{figure}[htbp]
\centering
%\bigskip

\begin{xy}
    
	\ar@{} (0,0);(0, -10)  *++! L{x_{1}} *\cir<4pt>{} = "X1"
	\ar@{} (0,0);(10, -10)  *++! L{x_{2}} *\cir<4pt>{} = "X2"
	\ar@{} (0,0);(30, -10)  *++! L{x_{2q - r + 1}} *\cir<4pt>{} = "X2q-r+1"
	\ar@{-} "X1";(0, 0) *++! L{y_{1}} *\cir<4pt>{} = "Y1";
	\ar@{-} "X2";(10, 0) *++!L{y_{2}} *\cir<4pt>{} = "Y2";
	\ar@{} (0, 0); (16, -5) *++!L{\cdots}
	\ar@{-} "X2q-r+1";(30, 0) *++!L{y_{2q - r + 1}} *\cir<4pt>{} = "Y2q-r+1";
    \ar@{-} "Y1"; (10, 10) *++!D{z_{2}} *\cir<4pt>{} = "Z2";
    \ar@{-} "Z2"; (0, 10) *++!D{z_{1}} *\cir<4pt>{} = "Z1";
    \ar@{-} "Z2"; (20, 10) *++!D{z_{3}} *\cir<4pt>{} = "Z3";
    \ar@{-} "Z3"; (30, 10) *++!D{z_{4}} *\cir<4pt>{} = "Z4";
    \ar@{-} "Z4"; (40, 10) *++!D{z_{5}} *\cir<4pt>{} = "Z5";
    \ar@{-} "Z5"; (50, 10) *++!D{z_{6}} *\cir<4pt>{} = "Z6";
    \ar@{-} "Z6"; (60, 10) *++!D{z_{7}} *\cir<4pt>{} = "Z7";
    \ar@{-} "Z7"; (70, 10) *++!D{z_{8}} *\cir<4pt>{} = "Z8";
    \ar@{-} "Z8"; (80, 10) *++!D{z_{9}} *\cir<4pt>{} = "Z9";
    \ar@{} "Z9"; (100, 10) *++!U{z_{2r - 2q - 2}} *\cir<4pt>{} = "Z2r-2q-2";
    \ar@{-} "Z2r-2q-2"; (110, 10) *++!D{z_{2r - 2q - 1}} *\cir<4pt>{} = "Z2r-2q-1";
    \ar@{-} "Z2r-2q-1"; (120, 10) *++!U{z_{2r - 2q}} *\cir<4pt>{} = "Z2r-2q";
    \ar@{-} "Z2r-2q"; (130, 10) *++!D{z_{2r - 2q + 1}} *\cir<4pt>{} = "Z2r-2q+1";
    \ar@{-} "Z2r-2q+1"; (140, 10) *++!U{z_{2r - 2q + 2}} *\cir<4pt>{} = "Z2r-2q+2";
    \ar@{-} "Z9"; (85, 10);
    \ar@{-} "Z2r-2q-2"; (95, 10);
    \ar@{--} (86, 10); (94, 10);
    \ar@{-} "Y2";"Z2";
    \ar@{-} "Y2q-r+1";"Z2";
    \ar@{-} "Z7"; (60, 0) *++!L{v_{1}} *\cir<4pt>{} = "V1";
    \ar@{-} "V1"; (60, -10) *++!L{u_{1}} *\cir<4pt>{} = "U1";
    \ar@{-} "Z9"; (80, 0) *++!L{v_{2}} *\cir<4pt>{} = "V2";
    \ar@{-} "V2"; (80, -10) *++!L{u_{2}} *\cir<4pt>{} = "U2";
    \ar@{-} "Z2r-2q-1"; (110, 0) *++!L{v_{r - q - 3}} *\cir<4pt>{} = "Vr-q-3";
    \ar@{-} "Vr-q-3"; (110, -10) *++!L{u_{r - q - 3}} *\cir<4pt>{} = "Ur-q-3";
    \ar@{-} "Z2r-2q+1"; (130, 0) *++!L{v_{r - q - 2}} *\cir<4pt>{} = "Vr-q-2";
    \ar@{-} "Vr-q-2"; (130, -10) *++!L{u_{r - q - 2}} *\cir<4pt>{} = "Ur-q-2";
    \ar@{} (0, 0); (90, -5) *++!L{\cdots}
\end{xy}

  \caption{The graph $G_{q, r}^{(2)}$}
  \label{fig:G_{q, r}^{(2)}}
\end{figure}

Now we prove that ${\rm ind}\mathchar`-{\rm match}\left( G_{q, r}^{(2)} \right) = {\rm min}\mathchar`-{\rm match}\left( G_{q, r}^{(2)} \right) = q$ and ${\rm match}\left( G_{q, r}^{(2)} \right) = r$.
\begin{itemize}
    \item We can see that  
\[
\left\{ \bigcup_{i = 1}^{2q - r + 1} \{x_i, y_i\} \right\} \ \cup \ \{z_4, z_5\} \ \cup \ \left\{ \bigcup_{i = 1}^{r - q  - 2} \{u_i, v_i\} \right\}
\]
    is an induced matching of $G_{q, r}^{(2)}$ and 
\[
\left\{ \bigcup_{i = 1}^{2q - r} \{x_i, y_i\} \right\} \ \cup \ \{y_{2q - r + 1}, z_{2}\} \ \cup \ \{z_4, z_5\} \ \cup \ \left\{ \bigcup_{i = 1}^{r - q  - 2} \{v_i, z_{2i + 5}\} \right\}
\]
    is a maximal matching of $G_{q, r}^{(2)}$. 
    Hence one has 
\[
q \leq {\rm ind}\mathchar`-{\rm match}\left( G_{q, r}^{(2)} \right) \leq {\rm min}\mathchar`-{\rm match}\left( G_{q, r}^{(2)} \right) \leq q. 
\]
    Thus ${\rm ind}\mathchar`-{\rm match}\left( G_{q, r}^{(2)} \right) = {\rm min}\mathchar`-{\rm match}\left( G_{q, r}^{(2)} \right) = q$ holds. 
    \item Since 
\[
\left\{ \bigcup_{i = 1}^{2q - r + 1} \{x_i, y_i\} \right\} \ \cup \ \left\{ \bigcup_{i = 1}^{r - q + 1} \{z_{2i - 1}, z_{2i}\} \right\} \ \cup \ \left\{ \bigcup_{i = 1}^{r - q  - 2} \{u_i, v_i\} \right\}
\]
    is a perfect matching of $G_{q, r}^{(2)}$, it follows that ${\rm match}\left( G_{q, r}^{(2)} \right) = r$. 
\end{itemize}
Hence one has 
\[
\min\left( q, q, r; |E| \right) \leq 
|E(G_{q, r})| = 2r - 1.
\]
Corollary \ref{LowerBound}(1) says that $\min\left( q, q, r; |E| \right) \geq 2r - 1$. 
Therefore $\min\left( q, q, r; |E| \right) = 2r - 1$. 

(3) We define the graph $G_{r}^{(3)}$ as follows; see Figure \ref{fig:G_{r}^{(3)}}: 
\begin{itemize}
    \item $\displaystyle \frac{}{} V\left( G_{r}^{(3)} \right) = X_{r} \cup Y_{r} \cup \{z\}$. 
    \item $\displaystyle E\left( G_{r}^{(3)} \right) = \left\{ \bigcup_{i = 1}^{r} \{x_i, y_i\} \right\} \ \cup \ \left\{ \bigcup_{i = 1}^{r} \{y_i, z\} \right\}$. 
\end{itemize}

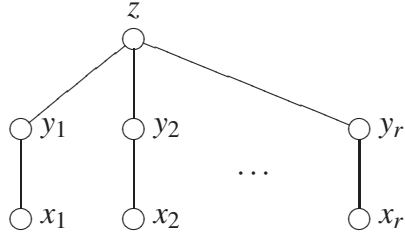
\begin{figure}[htbp]
\centering
%\bigskip

\begin{xy}
	\ar@{} (0,0);(50, -16)  *++! L{x_{1}} *\cir<4pt>{} = "X1"
	\ar@{} (0,0);(65, -16)  *++! L{x_{2}} *\cir<4pt>{} = "X2"
	\ar@{} (0,0);(95, -16)  *++! L{x_{r}} *\cir<4pt>{} = "Xr"
	\ar@{-} "X1";(50, -4) *++! L{y_{1}} *\cir<4pt>{} = "Y1";
	\ar@{-} "X2";(65, -4) *++!L{y_{2}} *\cir<4pt>{} = "Y2";
	\ar@{} (0, 0); (76, -10) *++!L{\cdots}
	\ar@{-} "Xr";(95, -4) *++!L{y_{r}} *\cir<4pt>{} = "Yr";
    \ar@{-} "Y1"; (65, 8) *++!D{z} *\cir<4pt>{} = "Z1";
    \ar@{-} "Y2";"Z1";
    \ar@{-} "Yr";"Z1";
\end{xy}

  \caption{The graph $G_{r}^{(3)}$}
  \label{fig:G_{r}^{(3)}}
\end{figure}

Then it is easy to see that ${\rm ind}$-${\rm match}\left( G_{r}^{(3)} \right) = {\rm min}$-${\rm match}\left( G_{r}^{(3)} \right) = {\rm match}\left( G_{r}^{(3)} \right) = r$. 
Hence one has 
\[
\min\left( r, r, r; |E| \right) \leq 
\left| E\left( G_{r}^{(3)} \right) \right| = 2r.
\]
By virtue of this fact together with Corollary \ref{LowerBound}(2), we have $\min\left( r, r, r; |E| \right) = 2r$. 
\end{proof} 

Next, we give upper bounds for $\min\left( 1, q, r; |E| \right)$ in the case that $2 \leq q \leq r \leq 2q - 2$. 

\begin{Theorem}\label{3rdMainThm}
Let $q, r$ be integers with $2 \leq q \leq r \leq 2q - 2$. Then
\begin{enumerate}
    \item[$(1)$] $\min\left( 1, q, q; |E| \right) \leq q^2$. 
    \item[$(2)$] $\min\left( 1, q, q + 1; |E| \right) \leq q^2 + 2$. 
    \item[$(3)$] $\min\left( 1, q, r; |E| \right) \leq \min\{ f_{1}(q, r), f_{2}(q, r) \}$ if $q + 2 \leq r \leq 2q - 2$, where 
    \[
    f_{1}(q, r) = r(q - 1) + \binom{r - q + 2}{2}, \ \ f_{2}(q, r) = 2(r - q) + \binom{2q}{2}. 
    \]
\end{enumerate}
\end{Theorem}
\begin{proof}
(1) It follows from that ${\rm ind}$-${\rm match}(K_{q, q}) = 1$, ${\rm min}$-${\rm match}(K_{q, q}) = {\rm match}(K_{q, q}) = q$ and $|E(K_{q, q})| = q^2$. 

(2) We define the graph $G_{q}^{(4)}$ as follows; see Figure \ref{fig:G_{q}^{(4)}}:
\begin{itemize}
    \item $\displaystyle \frac{}{} V\left( G_{q}^{(4)} \right) = X_{q} \cup Y_{q} \cup Z_{2}$. 
    \item $\displaystyle E\left( G_{q}^{(4)} \right) = \left\{ \bigcup_{1 \leq i, j \leq q} \{x_i, y_j\} \right\} \ \cup \ \{x_q, z_1\} \ \cup \ \{y_q, z_2\}$. 
\end{itemize}
Note that the induced subgraph $G_{q}^{(4)}[X_{q} \cup Y_{q}]$ is isomorphic to $K_{q, q}$. 

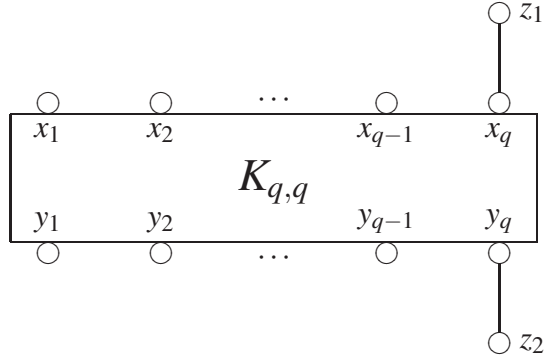
\begin{figure}[htbp]
\centering
%\bigskip

\begin{xy}
	\ar@{} (0,0);(45, -16)  *++! U{x_{1}} *\cir<4pt>{} = "X1";
	\ar@{} (0,0);(60, -16)  *++! U{x_{2}} *\cir<4pt>{} = "X2";
	\ar@{} (0,0);(90, -16)  *++! U{x_{q-1}} *\cir<4pt>{} = "Xq-1";
	\ar@{} (0,0);(75, -12) *++!U{\cdots};
	\ar@{} (0,0);(105, -16)  *++! U{x_{q}} *\cir<4pt>{} = "Xq";
    \ar@{-} "Xq";(105, -4) *++!L{z_{1}} *\cir<4pt>{} = "Z1";
    \ar@{} (0,0);(45, -36)  *++! D{y_{1}} *\cir<4pt>{} = "Y1";
    \ar@{} (0,0);(60, -36)  *++! D{y_{2}} *\cir<4pt>{} = "Y2";
    \ar@{} (0,0);(90, -36)  *++! D{y_{q-1}} *\cir<4pt>{} = "Yq-1";
    \ar@{} (0,0);(75, -33) *++!U{\cdots};
    \ar@{} (0,0);(105, -36)  *++! D{y_{q}} *\cir<4pt>{} = "Yq";
    \ar@{-} "Yq";(105, -48) *++!L{z_{2}} *\cir<4pt>{} = "Z2";
	\ar@{-} (40,-17.5); (110, -17.5);
	\ar@{-} (40,-17.5); (40, -34.5);
	\ar@{-} (40,-34.5); (110, -34.5);
	\ar@{-} (110,-17.5); (110, -34.5)
	\ar@{} (0,0);(75,-26.5) *{\text{\Large{$K_{q, q}$}}};
\end{xy}

  \caption{The graph $G_{q}^{(4)}$}
  \label{fig:G_{q}^{(4)}}
\end{figure}

Now we prove that ${\rm ind}\mathchar`-{\rm match}\left( G_{q}^{(4)} \right) = 1$, ${\rm min}\mathchar`-{\rm match}\left( G_{q}^{(4)} \right) = q$ and ${\rm match}\left( G_{q}^{(4)} \right) = q + 1$. 
\begin{itemize}
    \item Suppose that ${\rm ind}\mathchar`-{\rm match}\left( G_{q}^{(4)} \right) \geq 2$. 
    Then there exists an induced matching $I$ of $G_{q}^{(4)}$ with $|I| \geq 2$. 
    Since $V\left( G_{q}^{(4)} \right) \setminus N_{G_{q}^{(4)}}[x_{q}] = X_{q-1} \cup \{z_{2}\}$ is an independent set of $G_{q}^{(4)}$, one has $x_{q} \not\in V(I)$ by Lemma \ref{ind1-3}. 
    Similarly, we also have $y_{q} \not\in V(I)$. 
    Then $z_{1}, z_{2} \not\in V(I)$. 
    Hence $V(I) \subset X_{q-1} \cup Y_{q-1}$. In particular, $I$ is an induced matching of the induced subgraph $G_{q}^{(4)}[X_{q-1} \cup Y_{q-1}]$. 
    Thus one has 
    \[
    2 \leq |I| \leq {\rm ind}\mathchar`-{\rm match}\left( G_{q}^{(4)}[X_{q-1} \cup Y_{q-1}] \right) = {\rm ind}\mathchar`-{\rm match}\left( K_{q-1, q-1} \right) = 1, 
    \]
    but this is a contradiction. 
    Therefore we have ${\rm ind}\mathchar`-{\rm match}\left( G_{q}^{(4)} \right) = 1$. 
    \item Since $\bigcup_{i = 1}^{q} \{ x_{i}, y_{i} \}$ is a maximal matching of $G_{q}^{(4)}$, we have ${\rm min}\mathchar`-{\rm match}\left( G_{q}^{(4)} \right) \leq q$. 
    Moreover, by virtue of Lemma \ref{induced-subgraph}(2), one has 
    \[
    {\rm min}\mathchar`-{\rm match}\left( G_{q}^{(4)} \right) \geq {\rm min}\mathchar`-{\rm match}\left( G_{q}^{(4)}[X_{q} \cup Y_{q}] \right) = {\rm min}\mathchar`-{\rm match}\left( K_{q, q} \right) = q. 
    \]
    Hence ${\rm min}\mathchar`-{\rm match}\left( G_{q}^{(4)} \right) = q$ holds. 
    \item It follows that ${\rm match}\left( G_{q}^{(4)} \right) = q + 1$ because $\left\{ \bigcup_{i = 1}^{q-1} \{ x_{i}, y_{i} \} \right\} \cup \{x_{q}, z_{1}\} \cup \{y_{q}, z_{2}\}$ is a perfect matching of $G_{q}^{(4)}$. 
\end{itemize}
Therefore we have $\min\left( 1, q, q + 1; |E| \right) \leq q^2 + 2$ since $\left| E\left( G_{q}^{(4)} \right) \right| = q^2 + 2$. 

(3) First, we define the graph $G_{q,r}^{(5)}$ as follows; see Figure \ref{fig:G_{q,r}^{(5)}}:
\begin{itemize}
    \item $\displaystyle \frac{}{} V\left( G_{q,r}^{(5)} \right) = X_{q-1} \cup Y_{q-1} \cup Z_{r-q+1} \cup W_{r-q+1}$. 
    \item $\displaystyle E\left( G_{q,r}^{(5)} \right) = \left\{ \bigcup_{\substack{1 \leq i \leq q-1 \\ 1 \leq j \leq q-1}} \{x_i, y_j\} \right\} \ \cup \ \left\{ \bigcup_{\substack{1 \leq i \leq q - 1 \\ 1 \leq j \leq 2}} \{x_i, z_{r-q-1+j}\} \right\} \\ \cup \ \left\{ \bigcup_{\substack{1 \leq i \leq q-1 \\ 1 \leq j \leq r-q-1}} \{y_i, z_j\} \right\}  \ \cup \ \left\{ \bigcup_{1 \leq i < j \leq r-q+1} \{z_i, z_j\} \right\} \ \cup \ \left\{ \bigcup_{1 \leq i \leq r-q+1} \{z_i, w_i\} \right\}$. 
\end{itemize} 
Note that the induced subgraph $G_{q, r}^{(5)}[X_{q-1} \cup Y_{q-1}]$ is isomorphic to $K_{q-1, q-1}$ and the induced subgraph $G_{q, r}^{(5)}[Z_{r-q-1} \cup W_{r-q-1}]$ is isomorphic to the graph $G_{r-q+1}$ which appears in \ref{1.5}. 

\begin{figure}[htbp]
\centering
%\bigskip

\begin{xy}
	\ar@{} (0,0);(15, 0)  *++! D{x_{1}} *\cir<4pt>{} = "X1";
    \ar@{} (0,0);(25, 0)  *++! D{x_{2}} *\cir<4pt>{} = "X2";
    \ar@{} (0,0);(37.5, 3.5) *++!U{\cdots};
    \ar@{} (0,0);(50, 0)  *++! D{x_{2q-r}} *\cir<4pt>{} = "X2q-r";
    \ar@{} (0,0);(62.5, 3.5) *++!U{\cdots};
    \ar@{} (0,0);(75, 0)  *++! D{x_{q-1}} *\cir<4pt>{} = "Xq-1";
    \ar@{} (0,0);(95, 0)  *++! RD{z_{1}} *\cir<4pt>{} = "Z1";
    \ar@{} (0,0);(105, 0)  *++! RD{z_{2}} *\cir<4pt>{} = "Z2";
    \ar@{} (0,0);(115, 3.5) *++!U{\cdots};
    \ar@{} (0,0);(125, 0)  *++! L{z_{r-q-1}} *\cir<4pt>{} = "Zr-q-1";
    \ar@{} (0,0);(15, -30) *++! U{y_{1}} *\cir<4pt>{} = "Y1";
    \ar@{} (0,0);(25, -30) *++! U{y_{2}} *\cir<4pt>{} = "Y2";
    \ar@{} (0,0);(35, -26.5) *++!U{\cdots};
    \ar@{} (0,0);(45, -30) *++! U{y_{r-q-1}} *\cir<4pt>{} = "Yr-q-1";
    \ar@{} (0,0);(55, -30) *++! UL{y_{r-q}} *\cir<4pt>{} = "Yr-q";
    \ar@{} (0,0);(65, -26.5) *++!U{\cdots};
    \ar@{} (0,0);(75, -30) *++! U{y_{q-1}} *\cir<4pt>{} = "Yq-1";
    \ar@{} (0,0);(95, -30)  *++! R{z_{r-q}} *\cir<4pt>{} = "Zr-q";
    \ar@{} (0,0);(105, -30)  *++! L{z_{r-q+1}} *\cir<4pt>{} = "Zr-q+1";
    \ar@{-} "Z1";(95, 10) *++!D{w_{1}} *\cir<4pt>{} = "W1";
    \ar@{-} "Z2";(105, 10) *++!D{w_{2}} *\cir<4pt>{} = "W2";
    \ar@{-} "Zr-q-1";(125, 10) *++!D{w_{r-q-1}} *\cir<4pt>{} = "Wr-q-1";
    \ar@{-} "Zr-q";(95, -40) *++!R{w_{r-q}} *\cir<4pt>{} = "Wr-q";
    \ar@{-} "Zr-q+1";(105, -40) *++!L{w_{r-q+1}} *\cir<4pt>{} = "Wr-q+1";
%	\ar@{-} (10,-1.5); (130, -1.5);
%	\ar@{-} (10,-1.5); (10, -28.5);
%	\ar@{-} (10,-28.5); (130, -28.5);
%	\ar@{-} (130,-1.5); (130, -28.5);
    \ar@{-} "X1"; "Y1";
    \ar@{-} "X1"; "Y2";
    \ar@{-} "X1"; "Yr-q-1";
    \ar@{-} "X1"; "Yr-q";
    \ar@{-} "X1"; "Yq-1";
    \ar@{-} "X1"; "Zr-q";
    \ar@{-} "X1"; "Zr-q+1";
    \ar@{-} "X2"; "Y1";
    \ar@{-} "X2"; "Y2";
    \ar@{-} "X2"; "Yr-q-1";
    \ar@{-} "X2"; "Yr-q";
    \ar@{-} "X2"; "Yq-1";
    \ar@{-} "X2"; "Zr-q";
    \ar@{-} "X2"; "Zr-q+1";
    \ar@{-} "X2q-r"; "Y1";
    \ar@{-} "X2q-r"; "Y2";
    \ar@{-} "X2q-r"; "Yr-q-1";
    \ar@{-} "X2q-r"; "Yr-q";
    \ar@{-} "X2q-r"; "Yq-1";
    \ar@{-} "X2q-r"; "Zr-q";
    \ar@{-} "X2q-r"; "Zr-q+1";
    \ar@{-} "Xq-1"; "Y1";
    \ar@{-} "Xq-1"; "Y2";
    \ar@{-} "Xq-1"; "Yr-q-1";
    \ar@{-} "Xq-1"; "Yr-q";
    \ar@{-} "Xq-1"; "Yq-1";
    \ar@{-} "Xq-1"; "Zr-q";
    \ar@{-} "Xq-1"; "Zr-q+1";
    \ar@{-} "Z1"; "Y1";
    \ar@{-} "Z1"; "Y2";
    \ar@{-} "Z1"; "Yr-q-1";
    \ar@{-} "Z1"; "Yr-q";
    \ar@{-} "Z1"; "Yq-1";
%    \ar@{-} "Z1"; "Zr-q";
%    \ar@{-} "Z1"; "Zr-q+1";
    \ar@{-} "Z2"; "Y1";
    \ar@{-} "Z2"; "Y2";
    \ar@{-} "Z2"; "Yr-q-1";
    \ar@{-} "Z2"; "Yr-q";
    \ar@{-} "Z2"; "Yq-1";
%    \ar@{-} "Z2"; "Zr-q";
%    \ar@{-} "Z2"; "Zr-q+1";
    \ar@{-} "Zr-q-1"; "Y1";
    \ar@{-} "Zr-q-1"; "Y2";
    \ar@{-} "Zr-q-1"; "Yr-q-1";
    \ar@{-} "Zr-q-1"; "Yr-q";
    \ar@{-} "Zr-q-1"; "Yq-1";
%    \ar@{-} "Zr-q-1"; "Zr-q";
%    \ar@{-} "Zr-q-1"; "Zr-q+1";
    \ar@{--} (82,6); (143, 6);
    \ar@{--} (82,6); (82, -35);
    \ar@{--} (82,-35); (143, -35);
    \ar@{--} (143,6); (143, -35);
	\ar@{} (0,0);(120,-19) *{\text{  $G_{q,r}[Z_{r-q+1}] = K_{r-q+1}$}};
\end{xy}

  \caption{The graph $G_{q,r}^{(5)}$}
  \label{fig:G_{q,r}^{(5)}}
\end{figure}
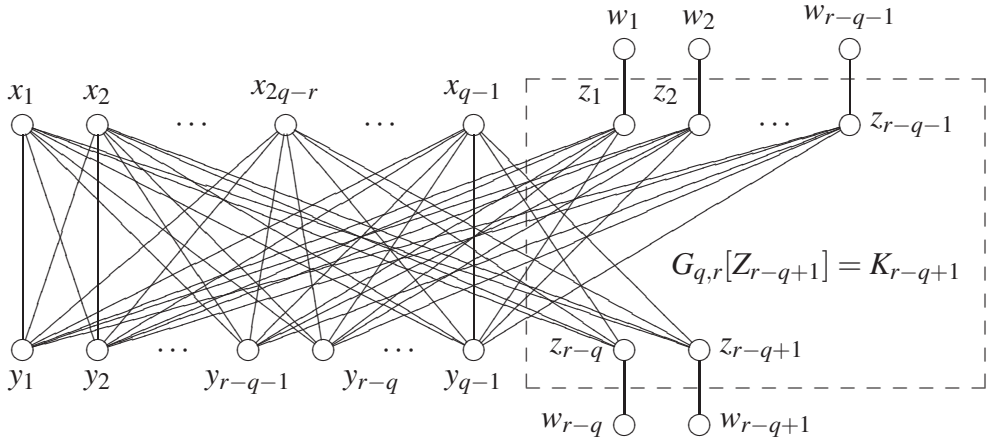

Now we prove ${\rm ind}\mathchar`-{\rm match}\left( G_{q,r}^{(5)} \right) = 1$, ${\rm min}\mathchar`-{\rm match}\left( G_{q,r}^{(5)} \right) = q$ and ${\rm match}\left( G_{q,r}^{(5)} \right) = r$.
\begin{itemize}
    \item Suppose that ${\rm ind}\mathchar`-{\rm match}\left( G_{q,r}^{(5)} \right) \geq 2$. 
    Then there exists an induced matching $I$ of $G_{q,r}^{(5)}$ with $|I| \geq 2$. 
    Since $V\left( G_{q,r}^{(5)} \right) \setminus N_{G_{q}^{(5)}}[z_{1}] = X_{q-1} \cup W_{r-q-1} \setminus \{w_{1}\}$ is an independent set of $G_{q,r}^{(5)}$, one has $z_{1} \not\in V(I)$ by Lemma \ref{ind1-3}. 
    Then $w_{1} \not\in V(I)$. 
    Similarly, we also have $z_{i}, w_{i} \not\in V(I)$ for all $2 \leq i \leq r - q - 1$.  
    Hence $V(I) \subset X_{q-1} \cup Y_{q-1}$. In particular, $I$ is an induced matching of the induced subgraph $G_{q,r}^{(5)}[X_{q-1} \cup Y_{q-1}]$. 
    Thus one has 
    \[
    2 \leq |I| \leq {\rm ind}\mathchar`-{\rm match}\left( G_{q,r}^{(5)}[X_{q-1} \cup Y_{q-1}] \right) = {\rm ind}\mathchar`-{\rm match}\left( K_{q-1, q-1} \right) = 1, 
    \]
    but this is a contradiction. 
    Therefore we have ${\rm ind}\mathchar`-{\rm match}\left( G_{q,r}^{(5)} \right) = 1$. 
    \item Since 
    \[
    \left\{ \bigcup_{i = 1}^{r-q-1} \{ y_{i}, z_{i} \} \right\} \ \cup \ \left\{ 
 \bigcup_{i = 1}^{2q - r} \{ x_{i}, y_{r-q-1+i} \} \right\} \ \cup \ \{z_{r-q}, z_{r-q+1}\}
    \]
    is a maximal matching of $G_{q,r}^{(5)}$, we have ${\rm min}\mathchar`-{\rm match}\left( G_{q,r}^{(5)} \right) \leq q$. 
    Note that the induced subgraph $G_{q,r}^{(5)}[X_{q-1} \cup Y_{q-1} \cup \{z_{1}, z_{r-q}\}]$ is isomorphic to $K_{q, q}$. 
    Hence, by using Lemma \ref{induced-subgraph}(2), one has 
    \begin{eqnarray*}
    {\rm min}\mathchar`-{\rm match}\left( G_{q,r}^{(5)} \right) &\geq& {\rm min}\mathchar`-{\rm match}\left( G_{q,r}^{(5)}[X_{q-1} \cup Y_{q-1} \cup \{z_{1}, z_{r-q}\}] \right) \\
    &=& {\rm min}\mathchar`-{\rm match}(K_{q, q}) \ = \ q. 
    \end{eqnarray*}
    Hence ${\rm min}\mathchar`-{\rm match}\left( G_{q,r}^{(5)} \right) = q$.  
    \item It follows that ${\rm match}\left( G_{q,r}^{(5)} \right) = r$ because $\left\{ \bigcup_{i = 1}^{q-1} \{ x_{i}, y_{i} \} \right\} \ \cup \ \left\{ \bigcup_{i = 1}^{r - q + 1} \{ z_{i}, w_{i} \} \right\}$ is a perfect matching of $G_{q,r}^{(5)}$. 
\end{itemize}
Since 
\begin{eqnarray*}
\left| E\left( G_{q, r}^{(5)} \right) \right| &=& (q-1)^2 + 2(q-1) + (q-1)(r-q-1) + \binom{r-q+2}{2}\\ 
&=& r(q-1) + \binom{r-q+2}{2},
\end{eqnarray*}
we have $\min\left( 1, q, r; |E| \right) \leq r(q-1) + \binom{r-q+2}{2}$. 

Next, let us consider the graph $G_{q, r-q, 0}^{(1)}$ which appears in \cite[1.3]{MY}. 
By virtue of \cite[Lemma 1.8]{MY}, it follows that  ind-match$\left( G_{q, r-q, 0}^{(1)} \right) = 1$, min-match$\left( G_{q, r-q, 0}^{(1)} \right) = q$ and match$\left( G_{q, r-q, 0}^{(1)} \right) = r$. 
Counting edges of $G_{q, r-q, 0}^{(1)}$, we have 
$\min\left( 1, q, r; |E| \right) \leq 2(r-q) + \binom{2q}{2}$. 
Therefore we have the desired conclusion. 
\end{proof}

\begin{Example}
Let $f_{1}(q, r)$ and $f_{2}(q, r)$ be functions appear in Theorem \ref{3rdMainThm}(3). 
\begin{enumerate} 
    \item[$(1)$] Since $f_{1}(10, 17) = 189$ and $f_{2}(10, 17) = 204$, we have $\min\left( 1, 10, 17; |E| \right) \leq 189$ by Theorem \ref{3rdMainThm}(3). 
    \item[$(2)$] Since $f_{1}(10, 18) = 207$ and $f_{2}(10, 18) = 206$, we have $\min\left( 1, 10, 18; |E| \right) \leq 206$ by Theorem \ref{3rdMainThm}(3).  
\end{enumerate}    
\end{Example}

Finally, we give upper bounds for $\min\left( p, q, r; |E| \right)$ in the case that $2 \leq p < q \leq r \leq 2q$ by utilizing Theorem \ref{Thm-ind} and Theorem \ref{Thm-min}. 

\begin{Theorem}\label{4thMainThm}
Let $p, q, r$ be integers with $2 \leq p < q \leq r \leq 2q$. Then
\begin{enumerate}
    \item[$(1)$] $\min\left( p, q, q; |E| \right) \leq (a_{1}^{2} + 1)p + (2a_{1} + 1)b_{1}$, where $a_{1}$ and $b_{1}$ are non-negative integers such that $q = a_{1}p + b_{1}$ and $0 \leq b_{1} \leq p-1$.  
    \item[$(2)$] If $q < r \leq 2q - p + 1$, we have 
    \[
    \min\left( p, q, r; |E| \right) \leq a_{2}^{2}(p - 1) + (2a_{2} + 1)b_{2} + p + \binom{2(r - q) + 1}{2}, 
    \]
    where $a_{2}$ and $b_{2}$ are non-negative integers such that $2q - r = a_{2}(p - 1) + b_{2}$ and $0 \leq b_{2} \leq p-2$.
    \item[$(3)$] If $2q - p + 1 < r \leq 2q$, we have
    \[
    \min\left( p, q, r; |E| \right) \leq p + 2q - r + (p - 2q + r)\binom{2a_{3}+1}{2} + b_3(4a_{3} + 3), 
    \]
    where $a_{3}$ and $b_{3}$ are non-negative integers such that $r - q = a_{3}(p - 2q + r) + b_{3}$ and $0 \leq b_{3} \leq p-2q + r - 1$.
\end{enumerate}
\end{Theorem}
\begin{proof}
(1) For each $1 \leq i \leq p$, let 
\[
H_{1} = \cdots = H_{p-b} = K_{a_{1}, a_{1}} \ \text{and} \ H_{p - b + 1} = \cdots = H_{p} = K_{a_{1}+1, a_{1} + 1}. 
\]
Note that there exists $v_{i} \in V(H_{i})$ which satisfies the condition $(*_{2})$ by Remark \ref{condition*_remark}(2). 
Let $G$ be the graph with 
\[
V(G) = \left\{ \bigcup_{i = 1}^{p} V(H_{i}) \right\} \ \cup \ \{v\} \ \ \text{and} \ \ E(G) = \left\{ \bigcup_{i = 1}^{p} E(H_{i}) \right\} \ \cup \ \left\{ \{v_{i}, v\} \mid 1 \leq i \leq p \right\}, 
\]
where $v$ is a new vertex.
Now we prove that ${\rm ind}\mathchar`-{\rm match}\left( G \right) = p$ and ${\rm min}\mathchar`-{\rm match}\left( G \right) = {\rm match}\left( G \right) = q$. 
\begin{itemize}
    \item Since ${\rm ind}\mathchar`-{\rm match}(H_{i}) = 1$ holds for each $1 \leq i \leq p$, one has ${\rm ind}\mathchar`-{\rm match}\left( G \right) = p$ by Theorem \ref{Thm-ind}. 
    \item By virtue of Theorem \ref{Thm-min}, we have
    \begin{eqnarray*}
    {\rm min}\mathchar`-{\rm match}(G) &=& \sum_{i = 1}^{p} {\rm min}\mathchar`-{\rm match}(H_{i}) \\
    &=& \sum_{i = 1}^{p-b_{1}} {\rm min}\mathchar`-{\rm match}(K_{a_{1}, a_{1}}) + \sum_{i = p - b + 1}^{b} {\rm min}\mathchar`-{\rm match}(K_{a_{1} + 1, a_{1} + 1}) \\
    &=& (p - b_{1})a_{1} + b_{1}(a_{1} + 1)  
    \ = \ a_{1}p + b_{1} \ = \ q. 
    \end{eqnarray*}
    \item Note that $H_{i}$ has a perfect matching $M_{i} \subset E(H_{i})$ for all $1 \leq i \leq p$. 
    Let $M = \bigcup_{i = 1}^{p} M_{i}$. 
    Then $M$ is a matching of $G$ with $|V(M)| = |V(G)| - 1$. 
    Hence one has 
    \[
    {\rm match}(G) = |M| = \sum_{i = 1}^{p} |M_{i}| = \sum_{i = 1}^{p} \frac{|V(H_{i})|}{2} = \sum_{i = 1}^{p-b_{1}} a_{1} + \sum_{i = p - b_{1} + 1}^{b_{1}} (a_{1} + 1) = q.
    \]
\end{itemize}
Therefore we have 
\begin{eqnarray*}
\min\left( p, q, q; |E| \right) \leq |E(G)| \ = \ \sum_{i = 1}^{p} |E(H_{i})| \ + \ p 
&=& \sum_{i = 1}^{p - b_{1}} a_{1}^{2} \ + \ \sum_{i = p - b_{1} + 1}^{p} \left( a_{1} + 1 \right)^{2} \ + \ p \\
&=& (p - b_{1})a_{1}^{2} + b_{1}\left( a_{1} + 1 \right)^{2} \ + \ p \\
&=& \left( a_{1}^{2} + 1 \right)p + (2a_{1} + 1)b_{1}. 
\end{eqnarray*}

(2) Assume that $q < r \leq 2q - p + 1$. 
Then $2q - r \geq p - 1$. 
Hence we can take non-negative integers $a_{2}, b_{2}$ such that $2q - r = a_{2}(p - 1) + b_{2}$ and $0 \leq b_{2} \leq p-2$. 
For each $1 \leq i \leq p$, let 
\[
H'_{1} = \cdots = H'_{p - b_{2} - 1} = K_{a_{2}, a_{2}}, \ \ H'_{p - b_{2}} = \cdots = H'_{p - 1} = K_{a_{2} + 1, a_{2} + 1}
\]
and $H'_{p} = G_{2(r-q)}$ which appears in \ref{1.5}. 
Note that there exists $v'_{i} \in V(H'_{i})$ which satisfies the condition $(*_{2})$ by Remark \ref{condition*_remark}(1),(2) and ${\rm ind}\mathchar`-{\rm match}(H'_{p}) = 1$ from Lemma \ref{G_r}. 
Let $G'$ be the graph with 
\[
V(G') = \left\{ \bigcup_{i = 1}^{p} V(H'_{i}) \right\} \ \cup \ \{v'\} \ \ \text{and} \ \ E(G') = \left\{ \bigcup_{i = 1}^{p} E(H'_{i}) \right\} \ \cup \ \left\{ \{v'_{i}, v'\} \mid 1 \leq i \leq p \right\}, 
\]
where $v'$ is a new vertex.
Now we prove that ${\rm ind}\mathchar`-{\rm match}\left( G' \right) = p$, ${\rm min}\mathchar`-{\rm match}\left( G' \right) = q$ and ${\rm match}\left( G' \right) = r$. 
\begin{itemize}
    \item As the same argument in (1), one has ${\rm ind}\mathchar`-{\rm match}\left( G' \right) = p$. 
    \item By virtue of Theorem \ref{Thm-min}, we have 
    \begin{eqnarray*}
    & & {\rm min}\mathchar`-{\rm match}(G') \\ 
    &=& \sum_{i = 1}^{p} {\rm min}\mathchar`-{\rm match}(H'_{i}) \\ 
    &=& \sum_{i = 1}^{p - b_{2} - 1} {\rm min}\mathchar`-{\rm match}(K_{a_{2}, a_{2}}) \ + \ \sum_{i = p - b_{2}}^{p - 1} {\rm min}\mathchar`-{\rm match}(K_{a_{2} + 1, a_{2} + 1}) \ + \ {\rm min}\mathchar`-{\rm match}\left( G_{2(r-q)} \right) \\ 
    &=& (p - b_{2} - 1)a_{2} \ + \ b_{2}(a_{2} + 1)\ + \ r - q \\
    &=& a_{2}(p - 1) + b_{2} + r - q \ = \ 2q - r + r - q \ = \ q. 
    \end{eqnarray*}
    \item As the same argument in (1), we have 
    \begin{eqnarray*}
    {\rm match}(G') \ = \ \sum_{i = 1}^{p} \frac{|V(H'_{i})|}{2} &=& (p - b_{2} - 1)a_{2} + b_{2}(a_{2} + 1) + 2(r - q) \\
    &=& a_{2}(p - 1) + b_{2} + 2(r - q) \\
    &=& 2q - r + 2(r - q) \ = \ r.
    \end{eqnarray*}
\end{itemize}
Therefore we have 
\begin{eqnarray*}
\min\left( p, q, r; |E| \right) \leq |E(G')| &=& \sum_{i = 1}^{p} |E(H'_{i})| \ + \ p \\
&=& \sum_{i = 1}^{p - b_{2} - 1} a_{2}^{2} \ + \ \sum_{i = p - b_{2}}^{p - 1} \left( a_{2} + 1 \right)^{2} \ + \ \binom{2(r-q) + 1}{2} \ + \ p \\
&=& (p - b_{2} - 1)a_{2}^{2} \ + \ b_{2}\left( a_{2} + 1 \right)^{2} \ + \ \binom{2(r-q) + 1}{2} \ + \ p \\
&=& a_{2}^{2}(p - 1) + (2a_{2} + 1)b_{2} + p + \binom{2(r - q) + 1}{2}. 
\end{eqnarray*} 

(3) Assume that $2q - p + 1 < r \leq 2q$. 
Then $p - 2q + r > 1$. 
Since $p < q$, one has $r - q - (p - 2q + r) = q - p > 0$. 
Hence we can take non-negative integers $a_{3}, b_{3}$ such that $r - q = a_{3}(p - 2q + r) + b_{3}$ and $0 \leq b_{3} \leq p - 2q + r - 1$. 
For each $1 \leq i \leq p$, let 
\[
H''_{1} = \cdots = H''_{p - 2q + r - b_{3}} = G_{2a_{3}}, \ \ H''_{p - 2q + r - b_{3} + 1} = \cdots = H''_{p - 2q + r} = G_{2a_{3} + 2}
\]
and $\displaystyle H''_{p - 2q + r + 1} = \cdots = H''_{p} = K_{2}$. 
Note that there exists $v''_{i} \in V(H''_{i})$ which satisfies the condition $(*_{2})$ by Remark \ref{condition*_remark}(1),(2) and ${\rm ind}\mathchar`-{\rm match}(H''_{p}) = 1$ from Lemma \ref{G_r}. 
Let $G''$ be the graph with 
\[
V(G'') = \left\{ \bigcup_{i = 1}^{p} V(H''_{i}) \right\} \ \cup \ \{v''\} \ \ \text{and} \ \ E(G'') = \left\{ \bigcup_{i = 1}^{p} E(H''_{i}) \right\} \ \cup \ \left\{ \{v''_{i}, v''\} \mid 1 \leq i \leq p \right\}, 
\]
where $v''$ is a new vertex.
Now we prove that ${\rm ind}\mathchar`-{\rm match}\left( G'' \right) = p$, ${\rm min}\mathchar`-{\rm match}\left( G'' \right) = q$ and ${\rm match}\left( G'' \right) = r$. 
\begin{itemize}
    \item As the same argument in (1), one has ${\rm ind}\mathchar`-{\rm match}\left( G'' \right) = p$. 
    \item By virtue of Theorem \ref{Thm-min}, we have 
    \begin{eqnarray*}
    & & {\rm min}\mathchar`-{\rm match}\left( G'' \right) \\ 
    &=& \sum_{i = 1}^{p} {\rm min}\mathchar`-{\rm match}\left( H''_{i} \right) \\
    &=& \sum_{i = 1}^{p - 2q + r - b_{3}} {\rm min}\mathchar`-{\rm match}\left( G_{2a_{3}} \right) \ + \ \sum_{i = p - 2q + r - b_{3} + 1}^{p - 2q + r} {\rm min}\mathchar`-{\rm match}\left( G_{2a_{3} + 2} \right) \\
    & &+ \ \sum_{p - 2q + r + 1}^{p} {\rm min}\mathchar`-{\rm match}\left( K_{2} \right) \\
    &=& (p - 2q + r - b_{3})a_{3} \ + \ b_{3}(a_{3} + 1) \ + \ 2q - r \\
    &=& a_{3}(p - 2q + r) + b_{3} + 2q - r \ = \ r - q + 2q - r \ = \ q. 
    \end{eqnarray*}
    \item As the same argument in (1), we have 
    \begin{eqnarray*}
    {\rm match}(G'') \ = \ \sum_{i = 1}^{p} \frac{|V(H''_{i})|}{2} &=& 2(p - 2q + r - b_{3})a_{3} \ + \ b_{3}(2a_{3} + 2) \ + \ 2q - r \\
    &=& 2\left\{ a_{3}(p - 2q + r) + b_{3} \right\} \ + 2q - r \\
    &=& 2(r - q) \ + \ 2q - r \ = \ r. 
    \end{eqnarray*}
\end{itemize}
Therefore one has 
\begin{eqnarray*}
\min\left( p, q, r; |E| \right) &\leq& |E(G'')| \\ &=& \sum_{i = 1}^{p} |E(H''_{i})| \ + \ p \\
&=& \sum_{i = 1}^{p - 2q + r - b_{3}} \binom{2a_{3} + 1}{2} \ + \ \sum_{i = p - 2q + r - b_{3} + 1}^{p - 2q + r} \binom{2a_{3} + 3}{2} \ + \ 2q - r + p \\
&=& (p - 2q + r - b_{3}) \binom{2a_{3} + 1}{2} \ + \ b_{3}\binom{2a_{3} + 3}{2} \ + \ 2q - r + p \\
&=& p + 2q - r + (p - 2q + r - b_{3}) \binom{2a_{3} + 1}{2} \ + \ b_{3}\binom{2a_{3} + 3}{2} \\
&=& p + 2q - r + (p - 2q + r) \binom{2a_{3} + 1}{2} + b_{3}(4a_{3} + 3). 
\end{eqnarray*}
\end{proof}

\begin{Example}\label{Ex2}
\begin{enumerate}
    \item[$(1)$] Since $p+1 = 1 \cdot p + 1$, one has $\min\left( p, p+1, p+1; |E| \right) \leq 2p + 3$ for all $p \geq 2$ by Theorem \ref{4thMainThm}(1). 
    \item[$(2)$] If $q = p + 1$ and $r = p + 2$, then $2q - p + 1 = p + 3 > r$. 
    Since $2q - r = p = 1 \cdot (p - 1) + 1$, we have $\min\left( p, p + 1, p + 2; |E| \right) \leq 2p + 5$ for all $p \geq 2$ by Theorem \ref{4thMainThm}(2). 
    \item[$(3)$] If $q = p + 1$ and $r = p + 4$, then $2q - p + 1 = p + 3 < r$. 
    Since $r - q = 3 = 1 \cdot 2 + 1$, one has $\min\left( p, p + 1, p + 4; |E| \right) \leq 2p + 11$ for all $p \geq 2$ by Theorem \ref{4thMainThm}(3).  
\end{enumerate}
\end{Example}

\section{Question}

Recall that ${\rm ind}\mathchar`-{\rm match}(G) \leq {\rm min}\mathchar`-{\rm match}(G)$ holds for all graph $G$ and a characterization of connected graphs $G$ with ${\rm ind}\mathchar`-{\rm match}(G) = {\rm min}\mathchar`-{\rm match}(G)$ is given \cite[Theorem 3.3]{HHKT}. 
It is interesting to find classes of connected graphs $G$ with ${\rm ind}\mathchar`-{\rm match}(G) = {\rm min}\mathchar`-{\rm match}(G)$. 

\begin{Question}\label{tree}
Does ${\rm ind}\mathchar`-{\rm match}(T) = {\rm min}\mathchar`-{\rm match}(T)$ hold for all tree $T$? 
\end{Question}

\begin{Theorem}
Let $p, q, r$ be integers with $2 \leq p \leq q \leq r \leq 2q$. 
Assume that Question \ref{tree} is true. 
Then one has
\begin{enumerate}
    \item[$(1)$] $\min\left( p, p+1, p+1; |E| \right) = 2p + 3$ holds. 
    \item[$(2)$] $\min\left( p, q, r; |E| \right)  = 2r - 1$ if and only if $p = q$. 
\end{enumerate}
\end{Theorem}
\begin{proof}
(1) Let $G$ be a connected graph with ${\rm ind}\mathchar`-{\rm match}(G) = p$ and ${\rm min}\mathchar`-{\rm match}(G) = {\rm match}(G) = p + 1$. 
Then we note that
\begin{itemize}
    \item $G$ is not a tree by assumption. Hence we have $|E(G)| \geq |V(G)|$. 
    \item $G$ does not have any perfect matching by Corollary \ref{notPM}. 
    Thus it follows that $|V(G)| \geq 2{\rm match}(G) + 1$.  
\end{itemize}
Hence one has $|E(G)| \geq 2{\rm match}(G) + 1 = 2p + 3$. 
Thus $\min\left( p, p+1, p+1; |E| \right) \geq 2p + 3$. 
Moreover, $\min\left( p, p+1, p+1; |E| \right) \leq 2p + 3$ holds from Example \ref{Ex2}(1). 
Therefore we have $\min\left( p, p+1, p+1; |E| \right) = 2p + 3$.  

(2) We remark that $\min\left( q, q, r; |E| \right) = 2r - 1$ holds from Theorem \ref{2ndMainThm}(2). 
Let $G$ be a connected graph with ${\rm ind}\mathchar`-{\rm match}(G) = p$, ${\rm min}\mathchar`-{\rm match}(G) = q$ and ${\rm match}(G) = r$. 
If $p \neq q$, then $G$ is not a tree by assumption. 
Hence one has $|E(G)| \geq |V(G)| \geq 2{\rm match}(G) = 2r$. 
Thus $\min\left( p, q, r; |E| \right) \geq 2r$ if $p \neq q$. 
Therefore we have the desired conclusion.
\end{proof}

\bigskip

\noindent
{\bf Acknowledgment.}
The first author was partially supported by JSPS Grants-in-Aid for Scientific Research (JP24K06661, JP24K14820, JP20KK0059).

\end{document}